\documentclass[10pt]{amsart}

\usepackage{amssymb,amsmath,amsthm}
\usepackage[all]{xy}
\usepackage{eucal}
\usepackage{verbatim}

\usepackage{color}

\theoremstyle{plain}
\newtheorem{thm}[subsection]{Theorem}

\newtheorem{prop}[subsection]{Proposition}

\theoremstyle{definition}
\newtheorem{defn}[subsection]{Definition}

\theoremstyle{remark}
\newtheorem{rem}[subsection]{Remark}

\makeatletter
\let\c@equation\c@subsection

\makeatother

\newcommand{\ZZ}{{ \mathbb{Z} }}

\newcommand{\tT}{{\tilde{T}}}
\newcommand{\tK}{{\tilde{K}}}

\newcommand{\capO}{{ \mathcal{O} }}

\newcommand{\ev}{{ \mathrm{ev} }}

\newcommand{\id}{{ \mathrm{id} }}
\newcommand{\op}{{ \mathrm{op} }}

\newcommand{\Space}{{ \mathsf{S} }}

\newcommand{\Ho}{{ \mathsf{Ho} }}

\newcommand{\sSet}{{ \mathsf{sSet} }}

\newcommand{\Mod}{{ \mathsf{Mod} }}

\newcommand{\M}{{ \mathsf{M} }}

\newcommand{\Alg}{{ \mathsf{Alg} }}

\newcommand{\TQ}{{ \mathsf{TQ} }}

\newcommand{\coAlg}{{ \mathsf{coAlg} }}
\newcommand{\res}{{ \mathsf{res} }}

\newcommand{\CGHaus}{{ \mathsf{CGHaus} }}

\newcommand{\AlgO}{{ \Alg_\capO }}

\newcommand{\Loop}{{ \Omega }}
\newcommand{\Loopt}{{ \tilde{\Omega} }}
\newcommand{\Sigmat}{{ \tilde{\Sigma} }}
\newcommand{\Susp}{{ \Sigma }}

\newcommand{\wequiv}{{ \ \simeq \ }}
\newcommand{\iso}{{ \cong }}

\DeclareMathOperator*{\hocolim}{hocolim}

\DeclareMathOperator*{\colim}{colim}
\DeclareMathOperator*{\holim}{holim}
\DeclareMathOperator{\Hombold}{\mathbf{Hom}}

\DeclareMathOperator{\Map}{Map}
\DeclareMathOperator{\End}{End}
\DeclareMathOperator{\BAR}{Bar}
\DeclareMathOperator{\Cobar}{Cobar}

\DeclareMathOperator{\Tot}{Tot}

\DeclareMathOperator{\im}{im}

\title[Iterated delooping and desuspension]{Iterated delooping and desuspension of structured ring spectra}

\author{Jacobson R. Blomquist}

\address{Department of Mathematical Sciences, 4400 Vestal Parkway E, Binghamton, New York 13902, USA}
\email{blomquist@math.binghamton.edu}

\begin{document}

\begin{abstract}

We study completion with respect to the iterated suspension functor on $\capO$-algebras, where $\capO$ is a reduced operad in symmetric spectra. This completion is the unit of a derived adjunction comparing $\capO$-algebras with coalgebras over the associated iterated suspension-loop homotopical comonad via the iterated suspension functor. We prove that this derived adjunction becomes a derived equivalence when restricted to 0-connected $\capO$-algebras and $r$-connected $\Sigmat^r \Loopt^r$-coalgebras. We also consider the dual picture, using iterated loops to build a cocompletion map from algebras over the iterated loop-suspension homotopical monad to $\capO$-algebras. This is the counit of a derived adjunction, which we prove is a derived equivalence when restricting to $r$-connected $\capO$-algebras and $0$-connected $\Loopt^r \Sigmat^r$-algebras.
\end{abstract}

\maketitle

\section{Introduction}



Homotopy groups are a powerful but difficult to work with invariant, so often throwing away information and focusing on homology is helpful because of the relative ease of computation. Because of this it is often very useful to have ways of comparing homotopy and homology groups. One such tool is Bousfield-Kan completion with respect to a ring, which essentially only sees the homology information of the space with respect to the given ring. It can be built out of iterating the integral chains functor and gluing everything together. The classical cases (Sullivan \cite{Sullivan_genetics}, Bousfield-Kan \cite{Bousfield_Kan}) were to use $R=\mathbb{F}_p$ or $R=\mathbb{Q}$ and these fit into a fracture square result where the space can be recovered out of a pullback of the $p$-completions and the rational completion. In the case of completing with respect to $\ZZ$, if the space is sufficiently nice (for example 1-connected) then the completion recovers the space up to homotopy equivalence \cite{Bousfield_Kan}. For a useful introduction to simplicial localization methods, see for instance \cite{Jardine} \cite{Goerss_Jardine}.

But this raises the question, how does gluing together homology information recover the homotopy type? The key is in realizing that homology carries more information than just the sequence of groups. The cosimplicial structure of the homology resolution encodes the data of being a coalgebra over the associated homology comonad, and the cosimplicial identities correspond to the coassociativity and counit axioms; the cobar is encoding that ``all diagrams'' involving these coalgebraic structure maps commute. The idea then is to study the category of colgebras over the associated comonad and see if this captures the entire homotopy category of spaces, after restricting to appropriate connectivity.

But the completion construction is fairly general, the next step is to ask what other comparison maps might work in place of the spaces level Hurewicz map to give interesting completions. Carlsson \cite{Carlsson_equivariant}, and later Arone-Kankaanrinta \cite{Arone_Kankaanrinta} prove an equivalence built out of the completion of a space with respect to stabilization $\id \to \Loopt^\infty \Sigma^\infty$ analogous to the Bousfield-Kan result. Hopkins \cite{Hopkins_iterated_suspension} and later Bousfield \cite{Bousfield_cosimplicial_space} prove the analogous equivalence coming from completion with respect to finite loop-suspension $\id \to \Loopt^r \Sigma^r$. Harper-Hess \cite{Harper_Hess} and then Ching-Harper \cite{Ching_Harper_derived_Koszul_duality} did the same for the analog to the homology case in the setting of $\capO$-algebras. In that case stabilization and $\TQ$-homology agree. The goal of this work is to consider completion in the finite suspension case in $\AlgO$ and show that it also fits into a derived equivalence, as well as analogous results for its dual, the cocompletion and how that relates to iterated loops.

To make these constructions homotopically meaningful in the context of $\capO$-algebras, we fatten things up between every step of the resolution, and the replacements we weave into the resolution need to be particularly well behaved. We work with homotopy (co)algebras and extend techniques to that context, showing in particular that we still have a proper (co)bar construction. For more about these ideas, see Blumberg-Riehl \cite{Blumberg_Riehl}.


To be able to talk about these (co)algebra categories representing homotopy theories, we need some notion of a homotopy theory on them. For $\AlgO$ this is just the one arising from the positive flat stable model structure. In the case of algebras we build this homotopy theory by hand using the box product to construct a topological $A_\infty$-category whose objects are algebras over the homotopical monad associated to loop-suspension and then take path components of the mapping spaces for the maps in the homotopy category. It is dual to work of Arone-Ching \cite{Arone_Ching_classification}.

\subsection{The (dual) Freudenthal suspension map}

From now on because we're in $\AlgO$, to keep everything homotopically meaningful we have to insert fibrant and cofibrant replacements, and for notational convenience will use $F$ for a simplicial fibrant replacement monad, $C$ for a simplicial cofibrant replacement comonad\cite{Blumberg_Riehl}, $\Loopt^r:=\Omega^r F$ and $\Sigmat^r:=\Sigma^rC$.

For a $k$-connected $\AlgO$ $X$, we have a result analogous to the Freudenthal suspension theorem for spaces which says that the composite
\begin{align}
CX \xrightarrow{m\id} CCX \xrightarrow{\id \eta \id^2} C\Loop^r \Sigmat^rX \xrightarrow{\id^2 \eta \id^3} C\Loopt^r\Sigmat^rX
\end{align}
of the comultiplication on $C$, followed by the unit on $\Loop^r\Sigma^r$, followed by the unit on $F$, is $(2k+2)$-connected. In the final chapter we use the classical spaces level version of this theorem to understand the homotopy theory of iterated suspension spaces, and as there is a version of this for spectra, we can do the same thing there.

And if we want to understand delooping, a good place to start would be the dual map
\begin{align*}
  F\Sigmat^r\Loopt^r X\xrightarrow{\id^2 \nu \id^3} F \Sigma^r \Loopt^r X \xrightarrow{\id \nu \id^2} FFX \xrightarrow{\mu \id} FX
\end{align*}
which is the counit on $C$ followed by the counit on $\Sigma^r\Loop^r$ followed by the multiplication on $F$. In \ref{dfsrs} we show that the connectivity of the above map is $2k+3-r$ for $X$ $k$-connected.

\subsection{Iterating the (dual) Freudenthal suspension map} Once we have maps like the ones above the natural thing to do is iterate them to form a (co)simplicial resolution of $X$. Compare with \cite{Bousfield_Kan}

Freudenthal
\begin{align}
\label{eq:freudenthal_resolution}
\xymatrix{
  CX \ar[r] & 
 C\Loopt^r  \Sigmat^rX \ar@<0.5ex>[r]\ar@<-0.5ex>[r] &
  C(\Loopt^r  \Sigmat^r)^2X \ar@<1.0ex>[r]\ar@<-1.0ex>[r]\ar@<0.0ex>[r]&
  C(\Loopt^r  \Sigmat^r)^3X \cdots
}  
\end{align}

Dual Freudenthal
\begin{align}\label{eq:dual_freudenthal_resolution}
\xymatrix{
  FX & 
F  \Sigmat^r\Loopt^r X\ar[l] & 
  F(\Sigmat^r\Loopt^r)^2 X\ar@<0.5ex>[l]\ar@<-0.5ex>[l] &
  F(\Sigmat^r\Loopt^r)^3\ar@<1.0ex>[l]\ar@<-1.0ex>[l]\ar@<0.0ex>[l] X\cdots
}  
\end{align}
We should think of the (co)simplicial structure as encoding the (co)-operations of the (co)monad. Bousfield \cite{Bousfield_cosimplicial_space} studied the spaces level version of \ref{eq:freudenthal_resolution}, the homotopy limit of which produces the $\Loopt^r \Sigmat^r$-completion map, analogous to the $\ZZ$-completion studied in \cite{Bousfield_Kan}.

\subsection{The Main Results}

In this section we will need Remarks \ref{frakC} and \ref{frakB} to understand the definitions of the cobar and bar constructions required for the following theorems.

\begin{thm}\label{dualmainthm}
The derived adjunction of the form
\begin{align}
  \Map_{\coAlg_{\Sigmat^r\Loopt^r}}(\Sigmat^rX, Y)\wequiv
  \Map_{\AlgO}(X, \holim\nolimits_{\Delta}\mathfrak{C}(Y))
\end{align}
induces an equivalence of homotopy categories, after restriction to the full subcategories of $0$-connected $\AlgO$ and $r$-connected $\Sigmat^r\Loopt^r$-coalgebras. More precisely:
\begin{itemize}
\item[(a)] 
$X\xrightarrow{\wequiv}\holim_{\Delta}\mathfrak{C}(\Sigmat^rX)$
\item[(b)]
$ \Sigmat^r \holim_{\Delta}\mathfrak{C}( Y)\xrightarrow{\wequiv} Y$
\end{itemize}

\end{thm}
\begin{thm}\label{mainthm}
The derived adjunction of the form
\begin{align}
  \Map_{\Alg_{\Loopt^r\Susp^r}}(X,\Loopt^r Y)\wequiv
  \Map_{\AlgO}(\hocolim\nolimits_{\Delta^\op}\mathfrak{B}(X),Y)
\end{align}
induces an equivalence of homotopy categories, after restriction to the full subcategories of $r$-connected $\AlgO$ and $0$-connected $\Loopt^r\Sigmat^r$-algebras. More precisely:
\begin{itemize}
\item[(a)] $  X\xrightarrow{\wequiv}\Loopt^r\hocolim_{\Delta^\op}\mathfrak{B}(X)$
\item[(b)]
$\hocolim_{\Delta^\op}\mathfrak{B}(\Loopt^r Y)\xrightarrow{\wequiv} Y$
\end{itemize}

\end{thm}

\subsection{Organization of the paper}

Section 2 gives the outline of the argument for the iterated suspension functor. In section 3 we prove the connectivity results required for the theorem. Section 4 gives the outline for the iterated loops result, and section 5 details the connectivity results required for this theorem. Section 6 contains some technical results required for the analysis to make everything homotopically meaningful. Section 7 shows how these techniques can also give an alternate proof to the equivalence of categories in \cite{Ching_Harper_derived_Koszul_duality} by calculating the appropriate connectivity estimates.

\subsection*{Acknowledgments}

The author would like to thank John E. Harper for consistent advice and countless helpful conversations. The author would also like to thank Cary Malkiewich for useful advice, Kathryn Hess for the very productive trip to EPFL, and Mark Johnson and Wojciech Dorabiala for stimulating discussion during his visit to Penn State Altoona. This research was supported in part by National Science Foundation grants DMS-1510640 and DMS-1547357.



\section{Proof of Theorem \ref{dualmainthm}}

This section is an outline of the proof of \ref{dualmainthm}.

\begin{thm}\label{lim1}
 For $X$ an $r$-connected $\Sigmat^r \Loopt^r$-coalgebra, $r>0$, $n \ge 0$ the map \begin{align} \Sigmat^r \holim_\Delta \mathfrak{C}(X) \to \Sigmat^r \holim_{\Delta^{\le n}} \mathfrak{C}(X) \end{align} is $(n+2+r)$-connected.
\end{thm}

\begin{proof}
By a formal argument, the connectivity of \begin{align} \holim_{\Delta^{\le n+1}} \mathfrak{C}(X) \to \holim_{\Delta^{\le n}} \mathfrak{C}(X) \end{align} is the same as the cartesianness of the $n$-fold iterated Freduenthal suspension map for $k=0$, which by \ref{hfsrs} is $(n+2)$-cartesian. But then if we increase $n$ this only increases, and so each such map is at least $(n+2)$-connected and therefore \begin{align} \holim_\Delta \mathfrak{C}(X) \to \holim_{\Delta^{\le n}} \mathfrak{C}(X) \end{align} is $(n+2)$-connected, and so applying $r$ fold suspension completes the proof.
\end{proof}

\begin{thm}\label{commutesigma}
For $X$ an $r$-connected $\Sigmat^r \Loopt^r$-coalgebra, $r>0$ the map \begin{align}\Sigmat^r \holim_{\Delta}\mathfrak{C}(X) \to \holim_{\Delta} \Sigmat^r \mathfrak{C}(X)\end{align} is an equivalence.
\end{thm}
\begin{proof}
Theorem \ref{lim1} and a standard $\lim \nolimits^1$ short exact sequence argument reduce this to showing that $$\Sigmat^r \holim_{\Delta^{\le n}} \mathfrak{C}(X) \to \holim_{\Delta^{\le n}}\Sigmat^r \mathfrak{C}(X)$$ has increasing connectivity in $n$. And by theorem \ref{sigmaestimates} that map is $(2n+r+7)$-connected.
\end{proof}

\begin{thm}
For $X$ an $r$-connected $\Sigmat^r \Loopt^r$-coalgebra, $ r>0$ the map \begin{align}\Sigmat^r \holim_{\Delta} \mathfrak{C}(X) \to X\end{align} is an equivalence.
\end{thm}

\begin{proof}
First recognize that $\Sigmat^r\mathfrak{C}(X)=\Cobar (\Sigmat^r \Loopt^r, \Sigmat^r \Loopt^r, X)$, which has an extra codegeneracy map, and therefore by the cofinality argument in Dror-Dwyer \cite[3.16]{Dror_Dwyer_long_homology} $$\holim_{\Delta}\Sigmat^r \mathfrak{C}(X) \to X$$ is an equivalence, and so \ref{commutesigma} completes the proof.
\end{proof}

\begin{thm}
For $Y$ a $0$-connected space the map \begin{align} Y \to \holim_{\Delta}\mathfrak{C}(\Sigmat^r Y) \end{align} is an equivalence.
\end{thm}

\begin{proof}
As above we filter by truncation. By a cofinality argument the desired connectivity of the $n$-th level is the cocartesianness of an iterated Freudenthal cube, whose uniform cartesianness we understand via higher Freudenthal Suspension (\ref{hfsrs}). Therefore by the immediately preceding theorem, we understand its cocartesianness, which is increasing in $n$. Another $\lim \nolimits^1$ short exact sequence argument completes the proof.
\end{proof}

\section{Connectivity Estimates for Freudenthal}

We proceed with a similar strategy. The following theorem is central to much of the analysis.

\begin{thm}[Freudenthal suspension for structured ring spectra]\label{fsrs} Let $k \ge 0$, $r \ge 1$. For $X$ a cofibrant $k$-connected $\mathcal{O}$-algebra, the map $X \to \Loopt^r \Sigmat^r X$ is $(2k+2)$-connected.
\end{thm}

\begin{proof}
Consider the maps $\Sigmat^r X \to \Sigmat^r \Loopt^r \Sigmat^r X \to \Sigmat^r X$, where the former is $\Sigmat^r$ applied to the Freudenthal suspension map for $X$, and the latter is the dual Freudenthal suspension map for $\Sigmat^r X$. By \ref{dfsrs} the latter map has connectivity $2(k+r)+3-r=2k+3+r$. Since the composition is the identity, it's infinitely connected so the standard connectivity result gives us that the latter map is $2k+2+r$ connected, which completes the proof.
\end{proof}

\begin{rem}
We could also prove this directly by using dual Blakers-Massey on a pushout square, dual to the proof of \ref{dfsrs}.
\end{rem}

\begin{prop}\label{blah}
An $n$-cube in $\mathcal{O}$-algebras is $(\id+1)(k+1)$-cartesian iff it is $((\id+1)(k+2)-2)$-cocartesian for $k\ge 0$.
\end{prop}

\begin{proof}
This is immediate from higher Blakers-Masey and higher dual Blakers-Massey for structured ring spectra \cite[1.7, 1.11]{Ching_Harper}, since the minimal partitions will be the trivial ones.
\end{proof}

\begin{thm}\label{hfsrs}
Let $\mathcal{X}$ be a $W$-cube in $\mathcal{O}$-algebras, $|W|=n$. For $k\ge 0$, $r \ge 1$ if $\mathcal{X}$ is $(\id+1)(k+1)$-cartesian, then the $(n+1)$-cube $\mathcal{X} \to \Loopt^r \Sigmat^r \mathcal{X}$ is $(\id+1)(k+1)$-cartesian.
\end{thm}

\begin{proof}
We induct on $n$. If $n=0$ this is the Freudenthal suspension theorem for structured ring spectra.

For $n>0$, without loss of generality, we may assume $\mathcal{X}$ is a cofibration cube. To make sure things are homotopically meaningful we rely on implicit use of \cite[I.2.5]{Chacholski_Scherer}. Let $C:=thocofib(\mathcal{X})$ and let $\mathcal{C}$ be the $W$-cube defined as $\mathcal{C}|_\emptyset =C$ and $\mathcal{C}|_T=*$ for $T \neq \emptyset$. Consider the following diagram of cubes:
\begin{align}
\xymatrix{
  \mathcal{X} \ar[r] \ar[d] & \Loopt^r \Sigmat^r\mathcal{X} \ar[d] \\
  \mathcal{C} \ar[r] &\Loopt^r \Sigmat^r \mathcal{C}
}
\end{align}

We want to know the cartesianness of the top cube, and we can do this by studying the other three. First, since $\mathcal{X}$ is $(\id+1)(k+1)$-cartesian by assumption, it is $((\id+1)(k+2)-2)$-cocartesian by the above theorem, and thus so is $\mathcal{C}$. Therefore, $C$ is $((n+1)(k+2)-2)$-connected, and so by Freudenthal suspension for structured ring spectra, $C \to \Loopt^r \Sigmat^r C$ is $(2(n+1)(k+2)-2)$-connected. Since the total fiber of $\mathcal{C} \to \Loopt^r \Sigmat^r \mathcal{C}$ is equivalent to the fiber of $\Loopt^n C \to \Loopt^{n+r} \Sigmat^{r} C$, the $(n+1)$ cube is $(2(n+1)(k+2)-n-2)$-cartesian. Note that $2(n+1)(k+2)-n-2 = (2n+2)(k+1)+n > (n+2)(k+1)$.

Next, we consider $\mathcal{X} \to \mathcal{C}$. This is an infinitely cocartesian cube, and we also know that a $d$-dimensional face of $\mathcal{X}$ is $((d+1)(k+2)-2)$-cocartesian. So a $(d+1)$-dimensional face is $((d+2)(k+2)-2)$-cocartesian, and a $(d+1)$-dimensional face of $\mathcal{X} \to \mathcal{C}$ made up of a map of $d$-dimensional faces of $\mathcal{C}$ and $\mathcal{X}$ is $((d+2)(k+2)-1)$-cocartesian. Therefore, since the minimal partition will be one of each of these two types of faces, by higher Blakers-Massey for structured ring spectra $\mathcal{X} \to \mathcal{C}$ is $((n+2)(k+1))$-cartesian.

Finally, to deal with $\Loopt^r \Sigmat^r \mathcal{X} \to \Loopt^r \Sigmat^r \mathcal{C}$, we first apply $\Sigmat^r$ to our above cocartesian estimates, to give faces of $((d+2)(k+2)+r-2)$-cocartesian and $((d+2)(k+2)+r-1)$-cocartesian, and for the same reasons higher Blakers-Massey for structured ring spectra tells us that $\Sigmat^r \mathcal{X} \to \Sigmat^r \mathcal{C}$ is $((n+2)(k+1)+2r)$-cartesian, and thus $\Loopt^r \Sigmat^r \mathcal{X} \to \Loopt^r \Sigmat^r \mathcal{C}$ is $((n+2)(k+1)+r)$-cartesian.

Therefore by the standard cartesianness properties we conclude that $\mathcal{X} \to \Loopt^r \Sigmat^r\mathcal{X}$ is $(n+2)(k+1)$-cartesian. Repeating this argument on all faces completes the proof.
\end{proof}

\begin{thm}\label{sigmaestimates}
Let $X$ be an $\Sigmat^r \Loopt^r$-coalgebra in $\mathcal{O}$-algebras, $k\ge r> 0$, and $n \ge 1$. Consider the infinitely cartesian $(n+1)$-cube $\widetilde{\mathfrak{C}(X)}$. If $X$ is $k$-connected then $\Sigmat^r \widetilde{\mathfrak{C}(X)}$ is $((k-r+1)(n+3)+r+1)$-cartesian.
\end{thm}

\begin{proof}
We know that one of the initial faces of $\widetilde{\mathfrak{C}(X)}$ is built from iterating the Freudenthal suspension map on $\Loopt^r X$, and thus by the higher Freudenthal suspension theorem for structured ring spectra is $(\id+1)(k-r+1)$-cartesian. We also know that $\widetilde{\mathfrak{C}(X)}$ is infinitely cartesian. Along with a common section argument \cite{Ching_Harper_derived_Koszul_duality}, higher dual Blakers-Massey shows that $\widetilde{\mathfrak{C}(X)}$ is $((n+3)(k-r+2)-2)$-cocartesian and another standard higher dual Blakers-Massey argument gets that it is also $((\id+1)(k-r+2)-2)$-cocartesian. Then $\Sigmat^r\widetilde{\mathfrak{C}(X)}$ is $((n+3)(k-r+2)-2+r)$-cocartesian and also $((\id+1)(k-r+2)-2+r)$-cocartesian immediately from the fact that $\Sigmat^r$ just increases cocartesianness by $r$. Finally, an application of higher Blakers-Massey for structured ring spectra completes the proof.
\end{proof}

\section{Proof of Theorem \ref{mainthm}}

This section is an outline of the proof of \ref{mainthm}.

\begin{thm}
For $X$ a $0$-connected $\Loopt^r \Sigmat^r$-algebra, $ r>0$ the map \begin{align}\hocolim_{\Delta^{\op}}\tilde \Omega^r \mathfrak{B}(X) \to \tilde \Omega^r \hocolim_{\Delta^{\op}} \mathfrak{B}(X)\end{align} is an equivalence.
\end{thm}
\begin{proof}
Because filtered colimits commute with homotopy groups, by filtering $\Delta^{\op}$ by truncation it suffices to show that $$\hocolim_{(\Delta^{\op})^{\le n}}\tilde \Omega^r \mathfrak{B}(X) \to \tilde \Omega^r \hocolim_{(\Delta^{\op})^{\le n}} \mathfrak{B}(X)$$ has increasing connectivity in $n$. And by theorem \ref{estimates} that map is $(3n+5)$-connected.
\end{proof}

\begin{thm}
For $X$ a $0$-connected $\Loopt^r \Sigmat^r$-algebra, $ r>0$ the map \begin{align}X \to \tilde \Omega^r \hocolim_{\Delta^{\op}} \mathfrak{B}(X)\end{align} is an equivalence.
\end{thm}

\begin{proof}
First recognize that $\tilde \Omega^r\mathfrak{B}(X)=\BAR(\tilde \Omega^r\Sigmat^r, \tilde \Omega^r\Sigmat^r, X)$, which has an extra degeneracy map, and therefore by a finality argument $$X \to \hocolim_{\Delta^{\op}}\tilde \Omega^r \mathfrak{B}(X)$$ is an equivalence, and so the above theorem completes the proof.
\end{proof}

\begin{thm}
For $Y$ an $r$-connected $\capO$-algebra, $r \ge 1$ the map \begin{align} \hocolim_{\Delta^\op}\mathfrak{B}(\Loopt^r Y)\to Y \end{align} is an equivalence.
\end{thm}

\begin{proof}
As above we filter by truncation. By a cofinality argument the desired connectivity of the $n$-th level is the cartesianness of an iterated Freudenthal cube, whose uniform cocartesianness we understand via Higher Dual Freudenthal Suspension (\ref{hdfsrs}). Therefore by the immediately preceding theorem, we understand its cartesianness, which is increasing in $n$. Commuting homotopy past the filtered colimit completes the proof.
\end{proof}

\section{Connectivity Estimates for Dual Freudenthal}

\begin{thm}[Dual Freudenthal suspension for structured ring spectra]\label{dfsrs} Let $k \ge r>0$. For $X$ a fibrant $k$-connected $\mathcal{O}$-algebra, the map $\Sigmat^r\Loopt^rX \to X$ is $(2k+3-r)$-connected.
\end{thm}

\begin{proof}
We proceed by induction. When $r=1$ this is an application of dual Blakers-Massey theorem for structured ring spectra\cite[1.9]{Ching_Harper} to the homotopy pullback square
\begin{align}
\xymatrix{
  \Loopt X \ar[r] \ar[d] & PX \ar[d] \\
  {*} \ar[r] & X
}
\end{align}

For $r>1$, we apply the $r=1$ case to the $(k-r+1)$-connected space $\Loopt^{r-1} X$ to get that $\Sigmat \Loopt^rX \to  \Loopt^{r-1}X$ is $(2k-2r+4)$-connected, and then suspending this map $r-1$ times gives us that $\Sigmat^r\Loopt^rX \to \Sigmat^{r-1}\Loopt^{r-1}X$ is $(2k+3-r)$-connected. By induction $\Sigmat^{r-1}\Loopt^{r-1}X \to X$ is $(2k+4-r)$-connected, and so their composition has the desired connectivity.
\end{proof}

\begin{prop}
An $n$-cube in $\mathcal{O}$-algebras is $((k+3-r)\id+k)$-cocartesian iff it is $((k+2-r)\id+k+1)$-cartesian for $k\ge r>0$.
\end{prop}

\begin{proof}
This is immediate from higher Blakers-Masey  and higher dual Blakers-Massey for structured ring spectra \cite[1.7, 1.11]{Ching_Harper}, since the minimal partitions will be the trivial ones.
\end{proof}

\begin{thm}\label{hdfsrs}
Let $\mathcal{X}$ be a $W$-cube in $\mathcal{O}$-algebras, $|W|=n$. For $k\ge r>0$, if $\mathcal{X}$ is $((k+3-r)\id+k)$-cocartesian, then the $(n+1)$-cube $\Sigmat^r\Loopt^r\mathcal{X} \to \mathcal{X}$ is $((k+3-r)\id+k)$-cocartesian.
\end{thm}

\begin{proof}
We induct on $n$. If $n=0$ this is the dual Freudenthal suspension theorem for structured ring spectra.

For $n>0$, without loss of generality, we may assume $\mathcal{X}$ is a fibration cube. Let $F:=thofib(\mathcal{X})$ and let $\mathcal{F}$ be the $W$-cube defined as $\mathcal{F}|_W=F$ and $\mathcal{F}|_T=*$ for $T \neq W$. Consider the following diagram of cubes:
\begin{align}
\xymatrix{
  \Sigmat^r\Loopt^r \mathcal{F} \ar[r] \ar[d] & \mathcal{F} \ar[d] \\
  \Sigmat^r\Loopt^r\mathcal{X} \ar[r] & \mathcal{X}
}
\end{align}

We want to know the cocartesianness of the bottom cube, and we can do this by studying the other three. First, since $\mathcal{X}$ is $((k+3-r)\id+k)$-cocartesian by assumption, it is $((k+2-r)\id+k+1)$-cartesian by the above theorem, and thus so is $\mathcal{F}$. Therefore, $F$ is $((k+2-r)n+k)$-connected, and so by dual Freudenthal suspension for structured ring spectra, $\Sigmat^r\Loopt^rF \to F$ is $((k+2-r)(2n)+2k+1)$-connected. Since the total cofiber of $\Sigmat^r\Loopt^r \mathcal{F} \to \mathcal{F}$ is equivalent to the fiber of $\Sigmat^{r+n}\Loopt^rF \to \Sigmat^nF$, it is $((k+2-r)(2n)+2k+(n+1))$-cocartesian. Note that $(k+2-r)(2n)+2k+(n+1) \ge (k+2-r)(n+1)+2k+(n+1)=(k+3-r)(n+1)+2k>(k+3-r)(n+1)+k$.

Next, we consider $\mathcal{F} \to \mathcal{X}$. This is an infinitely cartesian cube, and we also know that a $d$-dimensional face of $\mathcal{X}$ is $((k+2-r)d+k+1)$-cartesian. So a $(d+1)$-dimensional face is $((k+2-r)(d+1)+k+1)$-cartesian, and a $(d+1)$-dimensional face of $\mathcal{F} \to \mathcal{X}$ made up of a map of $d$-dimensional faces of $\mathcal{F}$ and $\mathcal{X}$ is $((k+2-r)d+k)$-cartesian. Therefore, since the minimal partition will be one of each of these two types of faces, by higher dual Blakers-Massey for structured ring spectra $\mathcal{F} \to \mathcal{X}$ is $((k+2-r)n+2k+2+n)$-cocartesian, and $(k+2-r)n+2k+2+n=(k+3-r)n+2k+2 \ge (k+2-r)n+k+(k+3-r)=(k+3-r)(n+1)+k$.

Finally, to deal with $\Sigmat^r \Loopt^r\mathcal{F} \to \Sigmat^r \Loopt^r\mathcal{X}$, we first apply $\Loopt^r$ to our above cartesian estimates, to give faces of $((k+2-r)(d+1)+k+1-r)$-cartesian and $((k+2-r)d+k-r)$-cartesian, and for the same reasons higher dual Blakers-Massey for structured ring spectra tells us that $ \Loopt^r\mathcal{F} \to  \Loopt^r\mathcal{X}$ is $((k+2-r)n+2k+2+n-2r)$-cocartesian, and thus $\Sigmat^r \Loopt^r\mathcal{F} \to \Sigmat^r \Loopt^r\mathcal{X}$ is $((k+2-r)n+2k+2+n-r)$-cocartesian. Note $(k+2-r)n+2k+2+n-r=(k+2-r)n+k+(k+2-r)+(n+1)-1=(k+2-r)(n+1)+k+(n+1)-1=(k+3-r)(n+1)+k-1$.

Therefore, we conclude that $\Sigmat^r\Loopt^r \mathcal{X} \to \mathcal{X}$ is $((k+3-r)(n+1)+k)$-cocartesian. Repeating this argument on all faces completes the proof.
\end{proof}

\begin{thm}\label{estimates}
Let $X$ be an $\Loopt^r \Sigmat^r$-algebra in $\mathcal{O}$-algebras, $k\ge 0$, and $r, n \ge 1$. Consider the infinitely cocartesian $(n+1)$-cube $\widetilde{\mathfrak{B}(X)}$. If $X$ is $k$-connected then $\Loopt^r \widetilde{\mathfrak{B}(X)}$ is $((k+3)(n+1)+2k+2)$-cocartesian
\end{thm}

\begin{proof}
We know that one of the initial faces of $\widetilde{\mathfrak{B}(X)}$ is built from iterating the dual Freudenthal suspension map on $\Sigmat^r X$, and thus by the higher dual Freudenthal suspension theorem for structured ring spectra is $((k+3)\id+k+r)$-cocartesian. We also know that $\widetilde{\mathfrak{B}(X)}$ is infinitely cocartesian. Along with a common section argument \cite{Ching_Harper_derived_Koszul_duality} this shows that $\widetilde{\mathfrak{B}(X)}$ is $((k+2)(n+1)+2k+2r+2)$- cartesian and a standard higher Blakers-Massey argument gets that it is also $((k+2)\id+k+r+1)$-cartesian. Then $\Loopt^r\widetilde{\mathfrak{B}(X)}$ is $((k+2)(n+1)+2k+r+2)$- cartesian and also $((k+2)\id+k+1)$-cartesian immediately from the fact that $\Loopt^r$ just decreases cartesianness by $r$. Finally, an application of higher dual Blakers-Massey for structured ring spectra completes the proof.
\end{proof}

\section{Technical Results}

\subsection{Suspension in $\AlgO$}

Quillen defines a notion of suspension for any model category via a homotopy pushout. But it will be helpful to have an explicit adjunction on the level of model categories.

Recall the action of $\Space_*$ on $\AlgO$:

\begin{defn}
For $\capO$ a reduced operad in $\mathcal{R}$-modules, $X,Y \in \AlgO$, $K \in \sSet_*$ define tensor hat $X \hat{\otimes} K$ in $\AlgO$ by the coequalizer 
\begin{align}
  X \hat{\otimes} K :=  \colim \Bigl(
  \xymatrix{
    \capO \circ (X \wedge K) &
    \capO \circ ((\capO \circ X) \wedge K)
    \ar@<0.5ex>[l]^(0.55){d_1}
    \ar@<-0.5ex>[l]_(0.55){d_0}
  }
  \Bigr)
\end{align}
in $\AlgO$ with $d_0=m \id \circ v$ for $v: (\mathcal{O} \circ X) \wedge K \to \mathcal{O} \circ (X \wedge K)$, which is induced by diagonal maps $K \to K^{\wedge n}$, analogously to the corresponding map in \cite{EKMM}, and $d_1=\id (m \wedge \id)$.

We also have a mapping object, $\Map(K, Y)$, defined as $\Map(\mathcal{R} \otimes G_0(K), Y)$ in $\Mod_\mathcal{R}$ with $\capO$-algebra structure induced from $Y$..
\end{defn}

The following theorem, analogous to \cite{EKMM} (see also \cite{Harper_Hess}) will be useful.

\begin{prop}\label{adjunction}
For $\capO$ a reduced operad in $\mathcal{R}$-modules, $X,  Y \in \AlgO$, $K \in \Space_*$, there is an isomorphism $$\hom_\AlgO (X \hat{\otimes} K, Y) \iso \hom_\AlgO(X, \Map (K, Y))$$ natural in $X, K, Y$.
\end{prop}

We begin with a technical result necessary for the above proposition. The proof is a formal argument.

\begin{prop}
Let $\capO$ be an operad in $\mathcal{R}$ modules, $X \in \Mod_\mathcal{R}$, $Y \in \AlgO$, and $K \in \Space_*$. For a map $f:X \wedge K \to Y$ in $\Mod_\mathcal{R}$ the diagram
\begin{align}
\xymatrix{
(\capO \circ X) \wedge K \ar[d]^{\nu} \ar[r]^(.41){\id \circ f \wedge \id} & (\capO \circ \Map (K, Y)) \wedge K \ar[r]^{\nu} & \capO \circ (\Map (K, Y) \wedge K) \ar[d]^{\id \circ \ev} \\
\capO \circ ( X \wedge K) \ar[rr]^{\id \circ f} & & \capO \circ Y
}
\end{align}
commutes in $\Mod_\mathcal{R}$.
\end{prop}

\begin{proof}[Proof of Proposition \ref{adjunction}]
First assume that $f:X \hat{\otimes} K \to Y$ is a map in $\AlgO$. Using $f$ to also denote $\capO \circ (X \wedge K) \to Y$ in $\AlgO$ and the adjoints $X \wedge K \to Y$ and $X \to \Map (K, Y)$ in $\Mod_\mathcal{R}$. The above proposition shows that
\begin{align}
\xymatrix{
(\capO \circ X) \wedge K \ar[r]^(.55){m \wedge \id} \ar[d]^{\id \circ f \wedge \id} & X \wedge K \ar[r]^(.38){f \wedge \id} & \Map (K, Y) \wedge K \ar[dd]^{\ev} \\
(\capO \circ \Map (K, Y)) \wedge K \ar[d]^{\nu} & & \\
\capO \circ (\Map (K, Y) \wedge K) \ar[r]^(.69){\id \circ \ev} & \capO \circ Y \ar[r]^{m} &  Y
}
\end{align}
commutes in $\Mod_\mathcal{R}$ and thus $f:X \to \Map(K, Y)$ is a map in $\AlgO$.

Now let $f: X \to \Map(K, Y)$ be a map in $\AlgO$ and consider $f:X \wedge K \to Y$ in $\Mod_\mathcal{R}$. We wish to show that the adjoint map $f:\capO \circ (X \wedge K) \to Y$ in $\AlgO$ induces a map $f: X \hat{\otimes} K \to Y$. Applying $\capO \circ (-)$ to the diagram in the above proposition does this and thus finishes the proof.
\end{proof}

For this construction to be homotopically useful, we need to know that the mapping object and tensor hat play nicely with the homotopy theory.

\begin{prop}
If $i:K \to L$ is a cofibration in $\Space_*$ and $p:X \to Y$ is a fibration in $\AlgO$, then the pullback corner map $$\Map(L, X) \to \Map (K, X) \times_{\Map (K, Y)} \Map (L, Y)$$ is a fibration, which is a weak equivalence if either $i$ or $p$ are.
\end{prop}

\begin{proof}
We use the fact that $\Mod_\mathcal{R}$ is a monoidal model category and that $\mathcal{R} \otimes G_0(-)$ preserves cofibrations.
\end{proof}

It's a quick formal proof to see that $- \hat{\otimes} S^1$ agrees with Quillen's notion of suspension.

\subsection{Homotopical Comonad}

Because of the required replacement functors, $\Sigmat^r \Loopt^r$ will not be an on the nose comonad. However, it will be a homotopical comonad, which has enough structure to give the required cobar construction. For more detail about these ideas, see \cite{Blumberg_Riehl}.

First, we show that $\Sigma^r C \Omega^r$ is a comonad. Given the adjunction \ref{adjunction}, we have a unit and counit map $$ \id \to \Omega^r \Sigma^r \quad \mathrm{and} \quad \Sigma^r \Omega^r \to \id$$ and the cofibrant replacement functor $C$ is a comonad, so has maps $$ C \to \id \quad \mathrm{and} \quad C \to CC$$ Therefore we can define $$\epsilon : \Sigma^r C \Omega^r \to \Sigma^r \id \Omega^r = \Sigma^r \Omega^r \to \id$$ and $$\nu : \Sigma^r C \Omega^r \to  \Sigma^r C C \Omega^r = \Sigma^r C \id C \Omega^r \to \Sigma ^r C \Omega^r  \Sigma^r C \Omega^r $$ which would need to satisfy the following diagrams to be a comonad

\begin{align}
\xymatrix{
\Sigma^r C \Omega^r \ar[r]^{\nu} \ar[d]^{\nu} & (\Sigma^r C \Omega^r)^2 \ar[d]^{\nu \id} & \Sigma^r C \Omega^r \ar[r]^{\nu} \ar[d]^{\nu} \ar[dr]^{\id} & \Sigma^r C \Omega^r \Sigma^r C \Omega^r \ar[d]^{\epsilon \id}\\
(\Sigma^r C \Omega^r)^2 \ar[r]^{\id \nu} & (\Sigma^r C \Omega^r)^3 &\Sigma^r C \Omega^r \Sigma^r C \Omega^r \ar[r]^{\id \epsilon} & \Sigma^r C \Omega^r
}
\end{align}

The left hand diagram, which represents coassociativity, commutes because the corresponding coassociativity diagrams commute for the comonads $ \Sigma^r\Omega^r$ and $C$, and because the comultiplication on $C$ and the unit on $ \Omega^r\Sigma^r$ can be done in any order, since they are disjoint.

The right hand diagram works out the same way, it comes from the corresponding counit diagrams and the fact that the maps interact nicely.

However, fibrant replacement, $F$, is a monad and therefore doesn't play as nicely, which means that we won't get a proper comonad. Though it will still have the appropriate properties to allow us to apply the bar construction, which is what we need. From here on to save notation, we will write $K:= \Sigma^r C\Loop^r$ and $\tK:=KF$.

The comultiplication map is defined as usual, \begin{align}\tilde{\nu}: KF \xrightarrow{\nu \id}  KKF \xrightarrow{\id \epsilon \id}  KFKF\end{align} but the counit map now has to be of the form \begin{align}\tilde \epsilon :KF \xrightarrow{\epsilon \id} F\end{align} and the diagrams that this satisfies are 

\begin{align}
\xymatrix{
\tK \ar[r]^{\tilde \nu} \ar[d]^{\tilde \nu} & \tK \tK \ar[d]^{\tilde \nu \id} & \tK \ar[r]^{\tilde \nu} \ar[dr]^{\id} & \tK \tK \ar[d]^{(*)} & F \tK \ar[r]^{\id \tilde \nu}  \ar[dr]^{\id} & F\tK \tK \ar[d]^{(**)}\\
\tK \tK \ar[r]^{\id \tilde \nu} & \tK \tK \tK & & \tK& &F \tK
}
\end{align}
where $(*)$ is the composite $KFKF \xrightarrow{\id^2 \tilde \epsilon } KFF \xrightarrow{\id \mu} KF$ and $(**)$ is $FKFKF \xrightarrow{\id \tilde \epsilon \id^2} FFKF \xrightarrow{\mu \id^2} FKF$.

Cossociativity still works as usual, as does the right counit map. But the left counit map requires an extra $F$ be carried around. But all three of these diagrams commute by similar arguments to the corresponding diagrams for $K$.

\begin{defn}
A homotopy $\tK$-coalgebra $Y$ is an $\capO$-algebra along with a map $m: Y \to \tK Y$ satisfying the following commutative diagrams:
\begin{align}
\xymatrix{
Y \ar[r]^{m} \ar[d]^{m} & \tK Y \ar[d]^{\tilde \nu \id} & FY \ar[r]^{\id m} \ar[dr]^{\id} &  F \tK Y \ar[d]^{(***)} \\
\tK Y \ar[r]^{ \id m} & \tK \tK Y & &  FY
}
\end{align}
where $(***)$ is the composite $FKFY \xrightarrow{\id \tilde \epsilon \id} FFY \xrightarrow{\mu \id} FY$.
\end{defn}

\begin{rem} \label{frakC}
It's easy to see that $C \Loopt^r$ has a right $\tK$-action, and so therefore for $Y$ a $\tK$-coalgebra, we can define the cobar construction as usual, with $\Cobar(C\Loopt^r , \tK, Y)_n=C\Loopt^r (\tK)^nY$, and all of the cosimplicial relations are still satisfied. We use the notation $\mathfrak{C}(Y):=\Cobar(C\Loopt^r, \tK, Y)$. Also, for an $\capO$-algebra $X$, $\Sigmat^rX$ is a homotopy $\tK$-coalgebra, and furthermore the Freudenthal map augments $\mathfrak{C}(\Sigmat^rX)$, giving us the resolution \ref{eq:freudenthal_resolution}.
\end{rem}

\subsection{Homotopical Monad}

What follows is the dual version of the previous section, to establish the required bar construction.

Again the first step is to show that $\Omega^r F \Sigma^r$ is a monad. Given the adjunction \ref{adjunction}, we have a unit and counit map $$ \id \to \Omega^r \Sigma^r \quad \mathrm{and} \quad \Sigma^r \Omega^r \to \id$$ and the fibrant replacement functor $F$ is a monad, so has maps $$ \id \to F \quad \mathrm{and} \quad F F \to F$$ Therefore we can define $$\eta : \id \to \Omega^r \Sigma^r = \Omega^r \id \Sigma^r \to \Omega^r F \Sigma^r$$ and $$\mu : \Omega^r F \Sigma^r \Omega^r F \Sigma^r \to \Omega^r F \id F \Sigma ^r = \Omega^r F F \Sigma^r \to \Omega^r F \Sigma^r$$ which would need to satisfy the following diagrams to be a monad

\begin{align}
\xymatrix{
(\Omega^r F \Sigma^r)^3 \ar[r]^{\id \mu} \ar[d]^{\mu \id} & (\Omega^r F \Sigma^r)^2 \ar[d]^{\mu} & \Omega^r F \Sigma^r \ar[r]^{\id \eta} \ar[d]^{\eta \id} \ar[dr]^{\id} & \Omega^r F \Sigma^r \Omega^r F \Sigma^r \ar[d]^{\mu}\\
(\Omega^r F \Sigma^r)^2 \ar[r]^{\mu} & \Omega^r F \Sigma^r &\Omega^r F \Sigma^r \Omega^r F \Sigma^r \ar[r]^{\mu} & \Omega^r F \Sigma^r
}
\end{align}

The left hand diagram, which gives associativity, commutes because the corresponding associativity diagrams commute for the monads $\Omega^r \Sigma^r$ and $F$, and because the multiplication on $F$ and the counit on $\Sigma^r \Omega^r$ can be done in any order, since they are disjoint.

The right hand diagram works out the same way, it comes from the corresponding unit diagrams and the fact that the maps interact nicely.

However, cofibrant replacement, $C$, is a comonad and therefore doesn't play as nicely, which means that we won't get a proper monad. This is alright though, it will still have the appropriate properties to allow us to apply the bar construction. From here on to save notation, we will write $T:=\Loopt^r \Sigma^r$ and $\tT:=TC$.

Our multiplication map is defined as usual, \begin{align} \tilde{\mu}: T C TC \xrightarrow{\id \epsilon \id}  T T C \xrightarrow{\mu \id}  T C\end{align} but the unit map now has to be of the form \begin{align}\tilde \eta :C \xrightarrow{\eta \id} T C\end{align} and the diagrams that this satisfies are 

\begin{align}
\xymatrix{
\tT \tT \tT \ar[r]^{\id \tilde \mu} \ar[d]^{\tilde \mu \id} & \tT \tT \ar[d]^{\tilde \mu} & \tT \ar[r]^{(*)} \ar[dr]^{\id} & \tT \tT \ar[d]^{\tilde \mu} & C \tT \ar[r]^{ (**)}  \ar[dr]^{\id} & C\tT \tT \ar[d]^{\id \tilde \mu}\\
\tT \tT \ar[r]^{ \tilde \mu} & \tT & & \tT& &C \tT
}
\end{align}
where $(*)$ is the composite $TC \xrightarrow{\id \nu} TCC \xrightarrow{\id \tilde \eta \id} TCTC$ and $(**)$ is $CTC \xrightarrow{\nu \id} CCTC \xrightarrow{\id \eta \id} CTCTC$.

Associativity still works as usual, as does the right unit map. But the left unit map requires an extra $C$ be carried around. But all three of these diagrams commute by similar arguments to the corresponding diagrams for $T$.

\begin{defn}
A homotopy $\tT$-algebra $X$ is an $\capO$-algebra along with a map $m:\tT X \to X$ satisfying the following commutative diagrams:
\begin{align}
\xymatrix{
\tT \tT X \ar[r]^{\id m} \ar[d]^{\tilde \mu \id} & \tT X \ar[d]^{m} & CX \ar[r]^{\nu \id} \ar[drr]^{\id} & CCX \ar[r]^{\id \tilde \eta \id} &  C \tT X \ar[d]^{\id m} \\
\tT X \ar[r]^{ m} & X & & &  CX
}
\end{align}
\end{defn}

\begin{rem} \label{frakB}
It's easy to see that $F \Sigmat^r$ has a right $\tT$-action, and so therefore for $X$ a $\tT$-algebra, we can define the bar construction as usual, with $\BAR(F\Sigmat^r , \tT, X)_n=F\Sigmat^r (\tT)^nX$, and all of the simplicial relations are still satisfied. We use the notation $\mathfrak{B}(X):=\BAR(F\Sigma^r C, \tT, X)$. Also, for an $\capO$-algebra $Y$, $\Loopt^rY$ is a homotopy $\tT$-algebra, and furthermore the dual Freudenthal map coaugments $\mathfrak{B}(\Loopt^rY)$, giving us the resolution \ref{eq:dual_freudenthal_resolution}.
\end{rem}

\subsection{Homotopy Theory of $\tT$-Algebras}

We will be defining the homotopy theory explicitly as an equivalence relation on a set of maps. This section is dual to \cite{Arone_Ching_classification}. The homotopy theory used for the $\tK$-coalgebras is completely analogous to \cite[Section 4]{Ching_Harper_derived_Koszul_duality}

\begin{defn}
A map of $\tT$-Algebras $X$, $Y$, is a map $f:X \to Y$ in $\AlgO$ such that the following diagram commutes
\begin{align}
\xymatrix{
\tT X \ar[r]^{m_X} \ar[d]^{\id f} & X \ar[d]^{f}  \\
\tT Y \ar[r]^{ m_Y} & Y
}
\end{align}
\end{defn}

\begin{defn}
It is easy to see that the set of maps $F:X \to Y$ that make the above diagram commute is isomorphic to the equalizer 
\begin{align}
\xymatrix{
 \lim( \hom_\AlgO (X,Y) \ar@<0.5ex>[r]^-{d_0}\ar@<-0.5ex>[r]_-{d_1} &  \hom_\AlgO (\tT X, Y))
}
\end{align}
where $d_0(f)=f \circ m_X$ and $d_1(f)=m_Y \circ \tT (f)$. We will call this $\hom_{\tT }(X,Y)$.
\end{defn}

\begin{rem}
By cofinality, this is equivalent to taking the limit over the whole cosimplicial set $\hom_\AlgO (C(\tT )^\bullet X, Y)$ with the extra cofibrant replacement there to make it homotopically meaningful. 
\end{rem}

Now to turn this into a homotopy class of maps, we need to derive it. We are able to understand our set of $\tT $-algebra maps as $\hom_{c\Space}(*, \Hombold(C(\tT )^\bullet A, A'))$, and therefore to get a homotopically meaningful mapping set we can just fatten our point, which is to say take the restricted totalization. Though for fibrancy reasons it will be helpful to work in $\CGHaus$. For this reason we recall the realization functor, $|-|:\Space \to \CGHaus$, and the Mapping space in $\CGHaus$ given by $$\Map_\AlgO(C(\tT )^\bullet A, A'):=|\Hombold_\AlgO(C(\tT )^\bullet A, A')|$$ This motivates the following definition.

\begin{defn}
Let $A$ and $A'$ be fibrant $\tT $-algebras, the mapping spaces of derived $\tT $-algebra maps are defined as the restricted totalizations $$\Hombold_{\Alg_{\tT }} (A, A'):=\Tot^\res \Hombold_\AlgO(C(\tT )^\bullet A, A')$$ $$\Map_{\Alg_{\tT }} (A, A'):=\Tot^\res \Map_\AlgO(C(\tT )^\bullet A, A')$$ 
\end{defn}

\begin{prop}
For $A$ and $A'$ fibrant $\tT$-algebras, the natural map $$|\Hombold_{\Alg_\tT}(A, A')| \to \Map_{\Alg_\tT}(A, A')$$ is a weak equivalence.
\end{prop}

\begin{proof}
This is an immediate consequence of the weak equivalence $$|\Tot^\res Y| \xrightarrow{\simeq} \Tot^\res |Y|$$ for objectwise fibrant $Y \in (\Space)^{\Delta_\res}$ which is proved in \cite[Prop 6.14]{Ching_Harper_derived_Koszul_duality}.
\end{proof}

It becomes helpful to have a language for working with the spaces of derived $\tT$-algebra maps, which the following definition provides.

\begin{defn}
For $A$ and $A'$ fibrant $\tT$-algebras, a derived $\tT $-algebra map $f: A \to A'$ is of the form $$f:\Delta[-] \to \Hombold_\AlgO (C(\tT )^\bullet A, A')$$ in $(\Space)^{\Delta_\res}$ and a topological derived $\tT $-algebra map $g: A \to A'$ is of the form $$g:\Delta^\bullet \to \Map_\AlgO (C(\tT )^\bullet A, A')$$ in $(\CGHaus)^{\Delta_\res}$. For a derived $\tT$-algebra map $f$ we refer to the map $f_0:CA \to A'$ which corresponds to $f_0: \Delta[0] \to \Hombold_\AlgO(CA, A')$ as the underlying map of $f$.
\end{defn}

\begin{rem}
Every derived $\tT$-algebra map $f$ determines a topological derived $\tT$-algebra map by realization.
\end{rem}


These mapping spaces aren't associative on the nose unfortunately, but they do have a highly homotopy coherent associativity structure. An important tool in realizing this structure is the box product. Recall the definition of the box product of two cosimplicial objects:

\begin{defn}
Let $(\M , \otimes )$ be a closed symmetric monoidal category. For $X,Y \in \M^\Delta$, the box product is defined \begin{align}
\xymatrix{
(X \Box Y )^n := \colim \Bigg( \displaystyle \coprod_{p+q=n-1} X^p \otimes Y^q \ar@<0.5ex>[r] \ar@<-0.5ex>[r] & \displaystyle \coprod_{p+q=n} X^p \otimes Y^q \Bigg)
}
\end{align}
where the two maps are $(d^{p+1}, \id)$ and $(\id, d^0)$ respectively.
\end{defn}

\begin{rem}
Because realization commutes with finite products and small colimits, for two cosimplicial spaces $X$ and $Y$ there is a natural isomorphism $|X \Box Y| \iso |X| \Box |Y|$ in $(\CGHaus)^{\Delta}$.
\end{rem}

\begin{defn}
For $A$, $A'$, and $A''$ $\tT $-algebras we define a composition map on the cosimplical spaces level $$\mu : \Hombold_\AlgO (C(\tT )^\bullet A, A') \Box \Hombold_\AlgO(C(\tT )^\bullet A', A'') \to \Hombold_\AlgO(C(\tT )^\bullet A, A'')$$ via the composites $$\xymatrix{ \Hombold_\AlgO (C(\tT )^n A, A') \times \Hombold_\AlgO(C(\tT )^m A', A'') \ar[d] \\ \Hombold_\AlgO (C(\tT )^m C(\tT )^n A, C(\tT )^m A') \times \Hombold_\AlgO (C(\tT )^m A', A'') \ar[d] \\ \Hombold_\AlgO(C(\tT )^{m+n} A, C(\tT )^m A') \times \Hombold_\AlgO(C(\tT )^m A', A'') \ar[d] \\ \Hombold_\AlgO (C(\tT )^{m+n} A, A'')}$$ and an identity map $$\iota : * \to \Hombold_\AlgO(C(\tT )^\bullet A, A)$$ by $$\Hombold_\AlgO(A, A) \to \Hombold_\AlgO(C(\tT )^nA, C(\tT )^nA) \to \Hombold_\AlgO (C(\tT )^n A, A)$$ where the latter map is a composition of $A$'s algebra structure map and the counit for $C$. We define corresponding maps $$\mu : \Map_\AlgO (C(\tT )^\bullet A, A') \Box \Map_\AlgO(C(\tT )^\bullet A', A'') \to \Map_\AlgO(C(\tT )^\bullet A, A'')$$ and $$\iota : * \to \Map_\AlgO(C(\tT )^\bullet A, A)$$ via realization.
\end{defn}

\begin{defn}
The coendomorphism operad $A$ is the (non-symmetric) operad in $\CGHaus$ defined objectwise by $$A(n):= \Map_{\Delta_\res}(\Delta^\bullet, (\Delta^\bullet)^{\Box n})$$
\end{defn}

\begin{prop}
For $n \ge 0$, $A_0,...A_n$ fibrant $\tT$-algebras the collection of maps $$A(n) \times \Map_{\Alg_\tT}(A_0,A_1) \times \cdots \times \Map_{\Alg_\tT}(A_{n-1},A_n)  \to \Map_{\Alg_\tT}(A_0,A_n)$$ induced by iterating $\mu$ determine a topological-$A_\infty$-category with objects the fibrant $\tT$-algebras and morphisms given by $\Map_{\Alg_\tT} (A, A')$.
\end{prop}

We then define the homotopy category by taking connected components.

\begin{defn}
The homotopy category of $\tT $-algebra spectra, denoted $\Ho(\Alg_{\tT })$, is the category with objects the fibrant $\tT $-algebras and for $A$ and $A'$ fibrant $\tT $-algebras, morphisms defined by $$[A, A']_{\tT }:=\pi_0 \Map_{\Alg_{\tT }}(A, A')$$
\end{defn}

\begin{prop}
A derived $\tT $-algebra map $f: A \to A'$ represents an isomorphism in the homotopy category iff the underlying map $f_0:CA \to A'$ is a weak equivalence.
\end{prop}

\section{Stable Estimates}

By the same methods as above, these techniques allow an alternate proof to \cite[Thm 1.11]{Ching_Harper_derived_Koszul_duality}. All we need is an analog of Theorem \ref{hfsrs}, given below, to get our connectivity estimates started and the rest goes through identically.

\begin{thm}[Higher Hurewicz for Structured Ring Spectra]
Let $\mathcal{X}$ be a $W$-cube in $\mathcal{O}$-algebras, $|W|=n$. For $k\ge 0$ if $\mathcal{X}$ is $(\id+1)(k+1)$-cartesian, then the $(n+1)$-cube $\mathcal{X} \to QU\mathcal{X}$ is $(\id+1)(k+1)$-cartesian.
\end{thm}

\begin{proof}
We induct on $n$. When $n=0$ this follows from \cite[1.8]{Harper_Hess}.

For $n>0$, without loss of generality, we may assume $\mathcal{X}$ is a cofibration cube. To make sure things are homotopically meaningful we rely on implicit use of \cite[I.2.5]{Chacholski_Scherer}. Let $C:=thocofib(\mathcal{X})$ and let $\mathcal{C}$ be the $W$-cube defined as $\mathcal{C}|_\emptyset =C$ and $\mathcal{C}|_T=*$ for $T \neq \emptyset$. Consider the following diagram of cubes:
\begin{align}
\xymatrix{
  \mathcal{X} \ar[r] \ar[d] & UQ\mathcal{X} \ar[d] \\
  \mathcal{C} \ar[r] &UQ \mathcal{C}
}
\end{align}

We want to know the cartesianness of the top cube, and we can do this by studying the other three. First, since $\mathcal{X}$ is $(\id+1)(k+1)$-cartesian by assumption, it is $((\id+1)(k+2)-2)$-cocartesian by proposition \ref{blah}, and thus so is $\mathcal{C}$. Therefore, $C$ is $((n+1)(k+2)-2)$-connected, and so by \cite[1.8]{Harper_Hess}, $C \to UQ C$ is $(2(n+1)(k+2)-2)$-connected. Since the total fiber of $\mathcal{C} \to UQ \mathcal{C}$ is equivalent to the fiber of $\Loopt^n C \to \Loopt^{n} UQ C$, the $(n+1)$ cube is $(2(n+1)(k+2)-n-2)$-cartesian. Note that $2(n+1)(k+2)-n-2 = (2n+2)(k+1)+n > (n+2)(k+1)$.

Next, we consider $\mathcal{X} \to \mathcal{C}$. This is an infinitely cocartesian cube, and we also know that a $d$-dimensional face of $\mathcal{X}$ is $((d+1)(k+2)-2)$-cocartesian. So a $(d+1)$-dimensional face is $((d+2)(k+2)-2)$-cocartesian, and a $(d+1)$-dimensional face of $\mathcal{X} \to \mathcal{C}$ made up of a map of $d$-dimensional faces of $\mathcal{C}$ and $\mathcal{X}$ is $((d+2)(k+2)-1)$-cocartesian. The minimal partition will be any partition into two faces, so by higher Blakers-Massey for structured ring spectra $\mathcal{X} \to \mathcal{C}$ is $((d+1)(k+2)-1)+((n-d+1)(k+2))-n-1=(n+2)(k+1)$-cartesian.

Finally, to deal with $UQ \mathcal{X} \to UQ \mathcal{C}$, we first note that $Q$ preserves cocartesianness and lands in a stable category, so $Q\mathcal{X} \to Q\mathcal{C}$ is infinitely cocartesian and thus infinitely cartesian. Then $U$ preserves cartesianness, so applying $U$, $UQ \mathcal{X} \to UQ \mathcal{C}$ is infinitely-cartesian.

Therefore by \cite[3.9]{Ching_Harper} we conclude that $\mathcal{X} \to UQ\mathcal{X}$ is $(n+2)(k+1)$-cartesian. Repeating this argument on all faces completes the proof.
\end{proof}

\bibliographystyle{plain}
\bibliography{DualFreudenthalSuspension}

\end{document}